\documentclass[preprint,12pt,3p]{elsarticle}
\usepackage{amsmath,amsthm,amsfonts,amssymb}
\usepackage{fullpage}
\usepackage{subfig}
\usepackage{blkarray}
\usepackage{caption}
\usepackage{float}
\usepackage{algpseudocode}
\usepackage{tikz}
\newcommand{\core}[1]{#1^{\tiny{\textcircled{\tiny\#}}}}

\usepackage{hyperref}
\numberwithin{equation}{section}
\usepackage{xcolor}
\usepackage{colortbl}
\usepackage[normalem]{ulem}
\usepackage{sectsty}

\definecolor{astral}{RGB}{46,116,181}
\subsectionfont{\color{astral}}
\sectionfont{\color{astral}}
\usepackage{colortbl}

\DeclareMathAlphabet{\mathpzc}{OT1}{pzc}{m}{it}
\DeclareFontFamily{OT1}{pzc}{}
\DeclareFontShape{OT1}{pzc}{m}{it}{<-> s * [0.900] pzcmi7t}{}
\DeclareMathAlphabet{\mathpzc}{OT1}{pzc}{m}{it}

\usepackage{amsbsy}
\usepackage{amsmath}
\usepackage{accents}
\newlength{\dhatheight}

\newenvironment{frcseries}{\fontfamily{frc}\selectfont}{}

\DeclareMathAlphabet\mathbfcal{OMS}{cmsy}{b}{n}
\definecolor{darkslategray}{rgb}{0.18, 0.31, 0.31}
\definecolor{warmblack}{rgb}{0.0, 0.26, 0.26}

\usepackage{multirow}

\usepackage{mathtools}
\usepackage{algorithm}
\usepackage{algorithmicx}
\makeatletter
\def\BState{\State\hskip-\ALG@thistlm}
\makeatother

\newtheorem{theorem}{Theorem}[section]
\newtheorem{lemma}[theorem]{Lemma}
\newtheorem{corollary}[theorem]{Corollary}
\theoremstyle{definition}
\usepackage{hyperref}
\newtheorem{definition}{Definition}[section]
\newtheorem{remark}{Remark}[section]
\newtheorem{example}{Example}[section]
\journal{...}
\newcommand{\R}{{\mathbb R}}

\begin{document}

\begin{frontmatter}

\title{Dual Drazin generalized inverse for dual matrices}

\vspace{-.4cm}
%

% \author{Vaibhav Shekhar$^b$ and Debasisha Mishra$^{*b}$}
\author{ Amit Kumar$^{1a}$ and Vaibhav Shekhar$^{2b}$}
 \address{

                   $^a$Department of Mathematics,\\
                        National Institute of Technology Raipur, India.\\

                        $^b$Department of Mathematics,\\ Indian Institute of Technology Delhi, India\\
                Email$^1$:amitdhull513@gmail.com\\
                Email$^2$:  vaibhavshekhar29@gmail.com}
                       
\vspace{-2cm}

\begin{abstract}
This manuscript proposes a generalized inverse for a dual matrix called dual Drazin generalized inverse (DDGI) which generalizes the notion of the dual group generalized inverse (DGGI). Under certain necessary and sufficient conditions, we establish the existence of the DDGI of a dual matrix of any index. Thereafter, we show that the DDGI is unique (whenever exists).  The DDGI is then used to solve a linear dual system. We also establish reverse-order law and forward-order law for a particular form of the DGGI, dual Moore-Penrose generalized inverse (DMPGI), dual core generalized inverse (DCGI), and DDGI under certain suitable conditions. Finally, the partial-orders based on DCGI and DGGI are proposed.
% The notion of dual Drazin generalized inverse (DDGI) for a dual matrix is introduced. Then the necessary and sufficient conditions for the existence of the DDGI is established. We also show that the DDGI are used to solve the linear dual system . At the end, we exploit a some sufficient condition under which  the  reverse and forward-order laws for a particular form  of the DGGI, DMPGI, DCGI, and DDGI  hold.
\end{abstract}


\end{frontmatter}

\newpage
\section{Introduction}
A representation of the form $\widehat{a}=a+\epsilon b$, where $a$ and $b$ are real numbers, and $\epsilon$ is the dual unit satisfying $\epsilon\neq 0$ and $\epsilon^2= 0$, is called a dual number. In this representation the number associated with $+1$ (i.e., $a$) and the dual unit $\epsilon$ (i.e., $b$) are called real and dual part of $\widehat{a}$, respectively.
The term dual number was first coined by William Clifford in the $18^{\text{th}}$ century, while Kotelnikov in 1895 was first one to show applications of dual numbers in mechanics (see \cite{rg1}). It is well known that the set of dual numbers forms a ring.
An extension of dual numbers are dual vectors and dual matrices.
Algebra of dual matrices has vast applications in many areas of science and engineering that includes kinematic problems \cite{app4}, robotics \cite{angeles}, surface shape analysis and computer graphics \cite{azar} and rigid body motion \cite{oncu}.

In this aspect, the use of generalized inverse to handle a case where we require to solve a dual linear system is shown in many recent articles \cite{udw1, udw2, rg4}. The study of generalized inverses of a dual matrix allow one to compute the dual equation simultaneously rather than  using the traditional method of splitting the system of dual linear equations into the real part and dual part and forming a system of real equations and then computing them individually. However, unlike the case where matrix has real entries, the existence of generalized inverses in case of dual number entries is not guaranteed. This makes it compulsory to investigate the conditions for existence of generalized inverses and to find their explicit expressions. We will now recall the definitions of some of the generalized inverses and their existence conditions available in the literature.
In 2009, Pennestr\'i and Valentini \cite{rg1, rg2} proposed the Moore-Penrose dual generalized inverse (MPDGI) of $\widehat{A}=A+\epsilon B$, denoted as $\widehat{A}^P$ and computed as $\widehat{A}^P=A^{\dagger}-\epsilon A^{\dagger}BA^{\dagger}$. The MPDGI is used in many inverse problems of kinematics and machines \cite{app1, app2, app3, app4}.
Udwadia {\it et al.} \cite{udw1, udw2} proposed the { dual Moore-Penrose generalized inverse} (DMPGI) which is recalled next.
Let $\widehat{A}=A+\epsilon B$ be a dual matrix, the unique matrix
$\widehat{X}$ (if it exists) satisfying
$$\widehat{A}\widehat{X}\widehat{A}=\widehat{A}, \widehat{X}\widehat{A}\widehat{X}=\widehat{X}, (\widehat{A}\widehat{X})^T=\widehat{A}\widehat{X}, \text{ and } (\widehat{X}\widehat{A})^T=\widehat{X}\widehat{A},$$
is called the DMPGI of $\widehat{A}$, and it is represented as $\widehat{A}^{\dagger}$.
DMPGI is applicable to the kinematics synthesis of spatial mechanisms \cite{udw1, udw2}. The authors \cite{udw1} showed that there are uncountably many dual matrices that do not have DMPGI.  For a dual matrix $\widehat{A}$, the smallest nonnegative integer for which $R(\widehat{A}^k)=R(\widehat{A}^{k+1})$ is called the {\it dual index} of the matrix $\widehat{A}$, and we denote it by $Ind(\widehat{A})$. If the index  of dual matrix $\widehat{A}=A+\epsilon B$ is $k$, then the index index of $A$ is also $k$, i.e., $Ind(A)=k$ (which can be prove in similar way as Theorem 2.1 of \cite{rg5}).
In 2021, Zhong and Zhang \cite{rg4} established the {\it dual group generalized inverse} (DGGI). The  definition of DGGI is stated next. Let $\widehat{A}$ be a dual matrix of size $n\times n$ with $Ind(\widehat{A})=1$. Then, the the unique dual matrix $\widehat{X}$ (if it exists) satisfying
$$\widehat{A}\widehat{X}\widehat{A}=\widehat{A}, \widehat{X}\widehat{A}\widehat{X}=\widehat{X}, \text{ and } \widehat{A}\widehat{X}=\widehat{X}\widehat{A},$$
is called the DGGI of $\widehat{A}$ and we denote it as $\widehat{A}^{\#}.$  The DGGI is used to find the dual $p$-norm solution of a dual linear system \cite{rg4}. Zhong and Zhang \cite{rg4} established the following necessary and sufficient conditions for the existence of the DGGI for a dual matrix.
\begin{lemma} (\cite{rg4})\label{rm1}\\
Let $\widehat{A}=A+\epsilon B$ be such that $A,B\in \mathbb{R}^{n\times  n},$ where $Ind(\widehat{A})=1$. Then, a dual matrix $\widehat{X}=X+\epsilon R$ of order  $n\times n$  is a dual inverse of $\widehat{A}$ iff $X=A^{\#}$ and
\begin{align*}
    B=AA^{\#}B+ARA+BA^{\#}A,\\
    R=A^{\#}AR+A^{\#}BA^{\#}+RAA^{\#},\\
    AR+BA^{\#}=RA+A^{\#}B.
\end{align*}
\end{lemma}

% The MPDGI is also used in many inverse problems of kinematics and machines \cite{app1, app2, app3, app4}.
% In 2009, Pennestr\'i and Valentini \cite{rg1, rg2} provided the Moore-Penrose dual generalized inverse (MPDGI), which is recalled next. Let $\widehat{A}=A+\epsilon B$, the MPDGI of $\widehat{A}$ is denoted by $\widehat{A}^P$ and it can be computed as $\widehat{A}^P=A^{\dagger}-\epsilon A^{\dagger}BA^{\dagger}$.  Udwadia {\it et al.} \cite{udw1, udw2} proposed the { dual Moore-Penrose generalized inverse} (DMPGI).  For a given  $\widehat{C},$ if there exists a dual matrix $\widehat{X}$ satisfying
% $$\widehat{C}\widehat{X}\widehat{C}=\widehat{C}, \widehat{X}\widehat{C}\widehat{X}=\widehat{X}, (\widehat{C}\widehat{X})^T=\widehat{C}\widehat{X}, \text{ and } (\widehat{X}\widehat{C})^T=\widehat{X}\widehat{C},$$
% then  $\widehat{X}$ is called the  DMPGI  of $\widehat{C}$, and it is  represented as $\widehat{C}^{\dagger}$. It is significance noting that for any dual matrix, its MPDGI always exists, while its DMPGI may not exist \cite{rg3}.

In 2022, Wang and Gao \cite{rg5} proposed the  {\it dual core generalized inverse} (DCGI) and provided the relation among the DCGI, DMPGI, DGGI and MPDGI. Next, we recall the definition of DCGI. Let $\widehat{A}$ be an  $n\times n$  dual matrix with $Ind(\widehat{A})=1$. An $n\times n$  dual matrix $\widehat{X}$ (if it exists) satisfying the following matrix equations
$$\widehat{A}\widehat{X}\widehat{A}=\widehat{A},~\widehat{A}\widehat{X}^2=\widehat{X} \text{ and } (\widehat{A}\widehat{X})^T=\widehat{A}\widehat{X},$$ is called a dual core generalized invertible matrix, and $\widehat{X}$ is the dual core generalized inverse of $\widehat{A}$, which is denoted as $\core{\widehat{A}}.$ Motivated by the works of \cite{udw1,udw2,rg5}, we introduce the dual Drazin generalized inverse (DDGI) for a square dual matrix and obtain its representation. We also establish its various properties and shown its applicability to solve a dual linear system. \par
For invertible matrices $A$ and $B$, the reverse-order law and forward-order law are represented as $(AB)^{-1}=B^{-1}A^{-1}$ and $(AB)^{-1}=A^{-1}B^{-1}$, respectively. 
The reverse and forward-order laws for different generalized inverses for elements and matrices are discussed (\cite{k11, k12, k13}). But it has yet to be examined for dual matrices. This manuscript also  presents sufficient conditions for reverse and forward-order laws to hold for the DMPGI, DGGI, DCGI, and DDGI. 
 
This manuscript is organized in this way:  Section \ref{rmm1}  recalls some preliminary results. In Section \ref{rm2} we present a dual generalized inverse (generalizing the notion of dual group inverse for arbitrary index), we call it dual Drazin generalized inverse (DDGI). After that, we investigate various properties of the DDGI. Section \ref{rm3} discusses some results of  solving a dual linear system. Section \ref{rm4} establishes the reverse and forward order laws  for DMPGI, DGGI, DCGI and DDGI. Lastly, we propose partial orders based on DGGI, DDGI and DCGI.
 %In section 4, we deal with a solving a linear dual system with the help of DDGI.
\section{Preliminaries}\label{rmm1}
In this section, we recall certain preliminary results. These results are helpful for proving the main results of this manuscript. First we recapture the rank relation for the Drazin inverse which is proved by Tian \cite{ks1}.
\begin{theorem}(Theorem 13.25, \cite{ks1})\label{kp1}\\
Let $A,B,C\in \mathbb{C}^{n \times n}$ with $Ind(B)=k$ and $Ind(C)=l$. Then,
$$rank\left(\begin{bmatrix}
A&B^k\\C^l&0
\end{bmatrix}\right)=rank(B^k)+rank(C^l)+rank\left[(I-BB^D)A(I-CC^D)\right].$$
\end{theorem}
The following expression to compute Drazin inverse will be used frequently.
\begin{theorem}\label{kp3}(\cite{ks1})\\ Let $A\in \mathbb{R}^{n\times n}$.
Then, $A$ has a Drazin  inverse iff there exist nonsingular matrices $P$ and $C$ such that \begin{align}\label{eq1}A=P\begin{bmatrix}
     C&0\\
     0&N
     \end{bmatrix}P^{-1}.\end{align}
 Furthermore,  \begin{align}\label{eq2}A^D=P\begin{bmatrix}
C^{-1}&0\\0&0
\end{bmatrix}P^{-1}.\end{align}
\end{theorem}
In 2021, Wang \cite{rg3}  provided certain conditions for the existence of the DMPGI of a dual matrix. The same is produced next.
\begin{theorem}(\cite{rg3})\label{kp2}\\
Let $\widehat{A}=A+\epsilon B$, where $Ind(A)=k$. Then, the following are equivalent:
\begin{enumerate}[(i)]
    \item The DMPGI $\widehat{A}^{\dagger}$ of $\widehat{A};$
    \item $(I_m-AA^{\dagger})B(I_n-A^{\dagger}A)=0;$
\item $rank\left(\begin{bmatrix}B&A\\A&0
\end{bmatrix}\right)=2 rank(A).$
\end{enumerate}
Furthermore, if $\widehat{A}^{\dagger}$ exists, then
$$\widehat{A}^{\dagger}=A^{\dagger}-\epsilon R,$$
where $R=A^{\dagger}BA^{\dagger}-(A^TA)^{\dagger}B^T(I_m-AA^{\dagger})-(I_n-A^{\dagger}A)B^T(AA^T)^{\dagger}.$
\end{theorem}

Similar result for group inverse is proved by Zhong and Zhang \cite{rg4}.
\begin{theorem}(\cite{rg4})\label{thm2}\\
    Let $\widehat{A}=A+\epsilon B$ be a dual matrix with $A,B\in\mathbb{R}^{n\times n}$, where $Ind(\widehat{A})=1$. Then, the following conditions are equivalent:
    \begin{enumerate}[(i)]
        \item The dual group inverse of $\widehat{A}$ exists;
        \item $\widehat{A}=P\begin{bmatrix}
            C&0\\0&0
        \end{bmatrix}P^{-1}+\epsilon P\begin{bmatrix}
            B_1&B_2\\B_3&0
        \end{bmatrix}P^{-1}$, where $C$ and $P$ are nonsingular matrices;
        \item $(I-AA^{\#})B(I-AA^{\#})=0$;
        \item $\begin{bmatrix}
            B&A\\A&0
        \end{bmatrix}^{\#}$ exists;
        \item $rank\begin{bmatrix}
            B&A\\A&0
        \end{bmatrix}=2rank(A)$\\
        
        Furthermore, if the dual group inverse of $\widehat{A}$ exists, then
        \begin{align}
            \widehat{A}^{\#}=A^{\#}+\epsilon R,
        \end{align}
        where $R=-A^{\#}BA^{\#}+(A^{\#})^2B(I-AA^{\#})+(I-AA^{\#})B(A^{\#})^2$.
        \end{enumerate}
\end{theorem}

% \begin{theorem}(Theorem 2.2, \cite{ks1})\label{kp2}\\
%  Let $\widehat{M} = M + \epsilon M_0$. Then, the following conditions are equivalent\begin{enumerate}
%      \item[(a)]  The DMPGI $\widehat{M}^{\dagger}$ of $M$ exists; 
%      \item[(b)] $rank\left(\begin{bmatrix}
%      M_0&A\\
%      A&0
%      \end{bmatrix}\right)=2rank(A)$.
%  \end{enumerate}
% \end{theorem}

%  Since they used the Drazin inverse and the Moore–Penrose inverse, this new generalized inverse is called the DMP inverse and defined as $A^{G,\dagger}=A^DAA^{\dagger}$
%  $A=U\begin{bmatrix}
%  \sum K&\sum L\\0&0
%  \end{bmatrix}U^T$
%  $A^{\dagger}=\begin{bmatrix}
%  K^T(\sum)^{-1}&0\\
%  L^T(\sum)^{-1}&0
%  \end{bmatrix}$
% $A=Q\begin{bmatrix}
% C&0\\0&N
% \end{bmatrix}Q^{-1}$ $A^D=Q^{-1}\begin{bmatrix}
% C^{-1}&0\\0&0
% \end{bmatrix}Q$.
\section{Dual Drazin generalized inverse (DDGI)}\label{rm2}
This section obtains the necessary and sufficient condition for the existence of the DDGI. We provide a characterization of the DDGI. Note that throughout this article, we consider the representation the matrix $A$ and its $A^D$ from Theorem \ref{kp3}. Now we start this section with a definition.
\begin{definition}
       Let $\widehat{A} = A + \epsilon B$ be such that $A, B \in\mathbb{ R}^{n\times n}$, where $ Ind(\widehat{A})= k$.  A dual  matrix $\widehat{G}$ (if it exists) is called the {\bf dual Drazin generalized inverse} (DDGI) of $\widehat{A}$ if it satisfies the conditions: 
\begin{eqnarray*}
\widehat{A}^{k}\widehat{X}\widehat{A} &= \widehat{A}^{k},\\~\widehat{X}\widehat{A}\widehat{X} &= \widehat{X},\\ \widehat{A}\widehat{X}& =\widehat{X}\widehat{A}.
\end{eqnarray*}
 It is denoted by $\widehat{X}=\widehat{A}^{D}$.
\end{definition}
Next result provides a necessary and sufficient condition for the existence of the DDGI. The proof can be easily obtained by following the above definition and is, therefore, skipped.

\begin{lemma}
Let $\widehat{A} = A + \epsilon B$ be such that $A, B \in\mathbb{ R}^{n\times n}$, where $ Ind(\widehat{A})= k$.
Then, a square $n\times n$ dual matrix $\widehat{X} = X+\epsilon R$ is a dual Drazin inverse of $\widehat{A}$ iff $X = A^{D}$ and
\begin{equation}\label{eqn3.1}
    A^kA^DB+A^kRA+(A^{k-1}B+A^{k-2}BA+\cdots+ABA^{k-2}+BA^{k-1})(A^DA-I)=0,
\end{equation}
\begin{equation}\label{eqn3.2}
    R=A^{D}AR+A^{D}BA^{D}+RAA^{D},
\end{equation}
\begin{equation}\label{eqn3.3}
    AR + BA^{D} = RA + A^DB.
\end{equation}
\end{lemma}
% \begin{corollary}(\cite{rg3})
% Let $\widehat{A} = A + \epsilon B$ be a dual matrix with $A, B \in R^{n\times n}$. If $=1$, then $\widehat{A}$ is dual group invertible. Moreover, $\widehat{X}=X+\epsilon R$ is dual group inverse, where $X=A^{\#}$ and
% \begin{equation}
%     B=AA^{\#}B+ARA+BA^{\#}A,
% \end{equation}
% \begin{equation}
%     R=A^{\#}AR+A^{\#}BA^{\#}+RAA^{\#},
% \end{equation}
% \begin{equation}
%     AR+BA^{\#}=RA+A^{\#}B.
% \end{equation}
% \end{corollary}
If $Ind(A)=1$, then the DDGI called the DGGI and the above lemma coincide with the Lemma \ref{rm1}. The dual group inverse doesn't exist for every square matrix. This is shown next.
\begin{example} Let
$\widehat{A}=\begin{bmatrix}
4&0&0\\
0&0&0\\
0&0&5
\end{bmatrix}+\epsilon
\begin{bmatrix}
1&0&4\\
1&2&0\\
0&2&0
\end{bmatrix}=A+\epsilon B$. Then, $A^{\#}=\begin{bmatrix}
0.25& 0&0\\0&0&0\\
0&0&0.2
\end{bmatrix}$. Suppose $R=\begin{bmatrix}
r_{11}&r_{12}&r_{13}\\
r_{21}&r_{22}&r_{23}\\
r_{31}&r_{32}&r_{33}
\end{bmatrix}$,  i.e., $AA^{\#}B+ARA+BA^{\#}A-B=\begin{bmatrix}
16r_{11}+1&0&12r_{13}+4\\
0&-2&0\\
12r_{31}&0&9r_{33}
\end{bmatrix}$. For the  $\widehat{A}$ dual group invertible must be $AA^{\#}B+ARA+BA^{\#}A-B=0$. But  $AA^{\#}B+ARA+BA^{\#}A-B\neq0$. Hence, $\widehat{A}$ is not group invertible.
\end{example}
% \begin{remark}
% We know that the class of group invertible matrices is subclass of Drazin invertible matrices class. Hence, the dual Drazin inverse doesn't exists for every square matrix.
% \end{remark}
The next result shows that the DDGI is unique whenever it exists.
\begin{theorem}
Let $\widehat{A} = A + \epsilon B$ be such that $A, B \in\mathbb{ R}^{n\times n}$, where $Ind(\widehat{A}) = k$. If $\widehat{A}^D$ exists, then it is unique.
\end{theorem}

\begin{proof}
Assume $\widehat{X} = X+\epsilon R_1$ and $\widehat{X_1} = X+\epsilon R_2$ be two dual Drazin inverses of $\widehat{A}$. Then, $R_1$ and $R_2$ satisfy the set of equations \eqref{eqn3.1}, \eqref{eqn3.2} and \eqref{eqn3.3}. Now, by \eqref{eqn3.1}, we have $A^k(R_1-R_2)A=0$. Post-multiplying by $(A^D)^{k+1}$, we get 
\begin{equation}\label{eqn3.4}
    A^k(R_1-R_2)A(A^D)^{k+1}=0.
\end{equation}
Equation \eqref{eqn3.3} gives 
\begin{equation}\label{eqn3.5}
    A(R_1-R_2)=(R_1-R_2)A.
\end{equation}
Using \eqref{eqn3.5} in \eqref{eqn3.4}, we get
$(R_1-R_2)A^kA(A^D)^{k+1}=0$ which further implies that 
\begin{equation}\label{eqn3.6}
    {AA^D(R_1-R_2)=(R_1-R_2)AA^D=0.}
\end{equation}
From \eqref{eqn3.2} and \eqref{eqn3.6}, 
\begin{equation*}
    R_1-R_2=AA^D(R_1-R_2)+(R_1-R_2)AA^D=0.
\end{equation*}
Hence, $\widehat{A}^D$ is unique.
\end{proof}
Next result is the primary result of this section in which we provide a representation of the DDGI. %equivalence relation. %DDGI and DGMPI. 

\begin{theorem}
Let $\widehat{A} = A+\epsilon B$ be such that $A, B \in\mathbb{ R}^{n\times n}$, where $Ind(\widehat{A}) = k$. Then, the following  are equivalent:
\begin{enumerate}[(i)]
\item $\widehat{A}^D$ exists;
   \item $\widehat{A}=P\begin{bmatrix}C &0\\
    0 &N\end{bmatrix}P^{-1}+\epsilon P\begin{bmatrix}B_1 &B_2\\B_3 &B_4\end{bmatrix}P^{-1}$, where $P$ and $C$ are nonsingular and
    $$N^{k-1}B_4+N^{k-2}B_4N+\cdots+B_4N^{k-1}=0;$$
    \item $(I-AA^D)(A^{k-1}B+A^{k-2}BA+\cdots+ABA^{k-2}+BA^{k-1})(A^DA-I)=0$;
    \item $rank\begin{bmatrix} D&A^k\\ A^k&0\end{bmatrix}=2rank(A^k)$,
    %  \item $\begin{bmatrix}A&D\\0 &A\end{bmatrix}^{D}$ exists, 
    where
    $$D=(A^{k-1}B+A^{k-2}BA+\cdots+ABA^{k-2}+BA^{k-1});$$
    \item  $\widehat{C}^{\dagger}$ exists, where $\widehat{C}=A^k+\epsilon D$. Further, if  $\widehat{A}^D$ exists, then
    $\widehat{A}^D=A^D+\epsilon R$, where\\
    $$R = -A^DBA^D+(A^D)^{(k+1)}D(I- AA^D) + (I-AA^D)D(A^D)^{(k+1)}.$$ 
\end{enumerate}
\end{theorem}
\begin{proof}
(i)$\implies$(ii):\\
Let $B=P\begin{bmatrix}B_1 &B_2\\B_3 &B_4
\end{bmatrix}P^{-1}$ and $R=P\begin{bmatrix}R_1 &R_2\\R_3 &R_4
\end{bmatrix}P^{-1}$ and suppose (i) holds, then using \eqref{eqn3.1}, clearly
\begin{align}\label{eqn3.8}
   &A^{k}A^DB+A^{k}RA+(A^{k-1}B+A^{k-2}BA+\cdots+ABA^{k-2}+BA^{k-1})(A^DA-I)=0,\notag
   \end{align}
   which implies that
{\scriptsize{ \begin{align*}  
&P\begin{bmatrix}C^k &0\\0 &0
\end{bmatrix}\begin{bmatrix}C^{-1} &0\\0 &0
\end{bmatrix}\begin{bmatrix}B_1 &B_2\\B_3 &B_4
\end{bmatrix}P^{-1}+P\begin{bmatrix}C^k &0\\0 &0
\end{bmatrix}\begin{bmatrix}R_1 &R_2\\R_3 &R_4
\end{bmatrix}\begin{bmatrix}C &0\\0 &N
\end{bmatrix}P^{-1}+ \Bigg(P\begin{bmatrix}C^{k-1} &0\\0 &N^{k-1}
\end{bmatrix}\begin{bmatrix}B_1 &B_2\\B_3 &B_4
\end{bmatrix}P^{-1}\\
&~~~~+P\begin{bmatrix}C^{k-2} &0\\0 &N^{k-2}
\end{bmatrix}\begin{bmatrix}B_1 &B_2\\B_3 &B_4
\end{bmatrix}\begin{bmatrix}C &0\\0 &N\end{bmatrix}P^{-1}+\dots+P\begin{bmatrix}B_1 &B_2\\B_3 &B_4\end{bmatrix}\begin{bmatrix}C^{k-1} &0\\0 &N^{k-1}\end{bmatrix}P^{-1}\Bigg)P
\begin{bmatrix}0&0\\0 &-I\end{bmatrix}P^{-1}=0,
\end{align*}}}
i.e.,
{\scriptsize{\begin{align}
&P\begin{bmatrix}C^{k-1}B_1 &C^{k-1}B_2\\0 &0
\end{bmatrix}P^{-1}+P\begin{bmatrix}C^{k}R_1C &C^{k}R_2N\\0 &0
\end{bmatrix}P^{-1}+\notag\\
&P\begin{bmatrix}C^{k-1}B_1+C^{k-2}B_1C+\cdots+B_1C^{k-1}&C^{k-1}B_2+C^{k-2}B_2N+\cdots+CB_2N^{k-2}+B_2N^{k-1}\\N^{k-1}B_3+N^{k-2}B_3C+\cdots+NB_3C^{k-2}+B_3C^{k-1}&N^{k-1}B_4+N^{k-2}B_4N+\cdots+NB_4N^{k-2}+B_4C^{k-1}
\end{bmatrix}\begin{bmatrix}0&0\\0 &-I\end{bmatrix}P^{-1}=0,\notag
\end{align}}}
i.e.,
{\scriptsize{\begin{align}
P\begin{bmatrix}C^{k-1}B_1+C^{k}R_1C &C^{k-1}B_2+C^{k}R_2N\\0 &0
\end{bmatrix}P^{-1}
+P\begin{bmatrix}0&-(C^{k-1}B_2+C^{k-2}B_2N+\cdots+CB_2N^{k-2}+B_2N^{k-1})\\0&-(N^{k-1}B_4+N^{k-2}B_4N+\cdots+NB_4N^{k-2}+B_4C^{k-1})
\end{bmatrix}P^{-1}=0,\notag
\end{align}}}
which further implies that $$C^{k-1}B_1+C^kR_1C=0,$$ $$C^{k-1}B_2+C^kR_2N-(C^{k-1}B_2+C^{k-2}B_2N+\cdots+CB_2N^{k-2}+B_2N^{k-1})=0,$$ and $$N^{k-1}B_4+N^{k-2}B_4N+\cdots+NB_4N^{k-2}+B_4C^{k-1}=0.$$\\
$(ii) \implies (iii)$: This can be shown by similar computations.\\
$(iii)\implies (i)$ By Theorem \ref{kp3}, there exist nonsingular matrices $P$ and $C$ such that $A$ and
$A^D$ are of the form \eqref{eq1} and \eqref{eq2}, respectively. Now, 
$$(I-AA^D)(A^{k-1}B+A^{k-2}BA+\cdots+ABA^{k-2}+BA^{k-1})(A^DA-I)=0,$$ implies that $N^{k-1}B_4+N^{k-2}B_4N+\cdots+NB_4N^{k-2}+B_4C^{k-1}=0.$ Assume 
{\scriptsize{$$\widehat{X}=P\begin{bmatrix}
C^{-1}&0\\0&0
\end{bmatrix}P^{-1}+\epsilon P\begin{bmatrix}
-C^{-1}B_1C^{-1}& C^{-2}B_2+C^{-3}B_2N+\cdots+C^{-(k+1)}N^{k-1}\\B_3C^{-2}+NB_3C^{-3}+\cdots+N^{k-1}B_3C^{-(k+1)}&0
\end{bmatrix}P^{-1}.$$}}
Further,
{\scriptsize{\begin{align*}
\widehat{A}\widehat{X}&=\left(P\begin{bmatrix}C &0\\
    0 &N\end{bmatrix}P^{-1}+\epsilon P\begin{bmatrix}B_1 &B_2\\B_3 &B_4\end{bmatrix}P^{-1}\right)\Bigg(P\begin{bmatrix}
C^{-1}&0\\0&0
\end{bmatrix}P^{-1}\\
&~~~+\epsilon P\begin{bmatrix}
-C^{-1}B_1C^{-1}& C^{-2}B_2+C^{-3}B_2N+\cdots+C^{-(k+1)}N^{k-1}\\B_3C^{-2}+NB_3C^{-3}+\cdots+N^{k-1}B_3C^{-(k+1)}&0
\end{bmatrix}P^{-1}\Bigg)\\
&=P\begin{bmatrix}
I&0\\0&0
\end{bmatrix}P^{-1}+\epsilon P\begin{bmatrix}
0&C^{-1}B_2+C^{-2}B_2N+\cdots+C^{-k}N^{k-1}\\B_3C^{-1}+NB_3C^{-2}+N^2B_3C^{-3}+\cdots+N^{k-1}B_3C^{-k}&0\end{bmatrix}P^{-1},
\end{align*}}}
 {\scriptsize{\begin{align*}
    \widehat{X}\widehat{A}&=
  \left(P\begin{bmatrix}
C^{-1}&0\\0&0
\end{bmatrix}P^{-1}+\epsilon P\begin{bmatrix}
-C^{-1}B_1C^{-1}& C^{-2}B_2+C^{-3}B_2N+\cdots+C^{-(k+1)}N^{k-1}\\B_3C^{-2}+NB_3C^{-3}+\cdots+N^{k-1}B_3C^{-(k+1)}&0
\end{bmatrix}P^{-1}\right)\\
&~~~~~\left(P\begin{bmatrix}C &0\\
    0 &N\end{bmatrix}P^{-1}+\epsilon P\begin{bmatrix}B_1 &B_2\\B_3 &B_4\end{bmatrix}P^{-1}\right)\\
&=P\begin{bmatrix}
I&0\\0&0
\end{bmatrix}P^{-1}+\epsilon P\begin{bmatrix}
0&C^{-1}B_2+C^{-2}B_2N^{1}+\cdots+C^{-k}N^{k-1}\\B_3C^{-1}+NB_3C^{-2}+N^2B_3C^{-3}+\cdots+N^{k-1}B_3C^{-k}&0\end{bmatrix}P^{-1},
\end{align*}}}

{\scriptsize{\begin{align*}
    \widehat{X}\widehat{A}\widehat{X}&= \left(P\begin{bmatrix}
C^{-1}&0\\0&0
\end{bmatrix}P^{-1}+\epsilon P\begin{bmatrix}
-C^{-1}B_1C^{-1}& C^{-2}B_2+C^{-3}B_2N+\cdots+C^{-(k+1)}N^{k-1}\\B_3C^{-2}+NB_3C^{-3}+\cdots+N^{k-1}B_3C^{-(k+1)}&0
\end{bmatrix} P^{-1}\right)\\
&~~~~~~\left(P\begin{bmatrix}C &0\\
    0 &N\end{bmatrix}P^{-1}+\epsilon P\begin{bmatrix}B_1 &B_2\\B_3 &B_4\end{bmatrix}P^{-1}\right)
   \Bigg(P\begin{bmatrix}
C^{-1}&0\\0&0
\end{bmatrix}P^{-1}+\\
&~~~~~\epsilon P\begin{bmatrix}
-C^{-1}B_1C^{-1}& C^{-2}B_2+C^{-3}B_2N^{1}+\cdots+C^{-(k+1)}N^{k-1}\\B_3C^{-2}+NB_3C^{-3}+\cdots+N^{k-1}B_3C^{-(k+1)}&0
\end{bmatrix}P^{-1}\Bigg)\\
&=\left(P\begin{bmatrix}
C^{-1}&0\\0&0
\end{bmatrix} P^{-1}+\epsilon P\begin{bmatrix}
-C^{-1}B_1C^{-1}& C^{-2}B_2+C^{-3}B_2N+\cdots+C^{-(k+1)}N^{k-1}\\B_3C^{-2}+NB_3C^{-3}+\cdots+N^{k-1}B_3C^{-(k+1)}&0
\end{bmatrix}P^{-1}\right)\\
&~~~~~~P\begin{bmatrix}
I&0\\0&0
\end{bmatrix}P^{-1}+\epsilon P\begin{bmatrix}
0&C^{-1}B_2+C^{-2}B_2N+\cdots+C^{-k}N^{k-1}\\B_3C^{-1}+NB_3C^{-2}+N^2B_3C^{-3}+\cdots+N^{k-1}B_3C^{-k}&0\end{bmatrix}P^{-1}\\
&=P\begin{bmatrix}
C^{-1}&0\\0&0
\end{bmatrix}P^{-1}+\epsilon P\begin{bmatrix}-C^{-1}BC^{-1}&C^{-2}B_2+C^{-3}B_2N+\cdots+C^{-(k+1)}N^{k-1}\\B_3C^{-2}+NB_3C^{-3}+\cdots+N^{k-1}B_3C^{-(k+1)}&0
\end{bmatrix}P^{-1}\\
&=\widehat{X},
\end{align*}}}

 and

{\scriptsize{\begin{align*}
\widehat{A}^k\widehat{X}\widehat{A}&=\Bigg(P\begin{bmatrix}
    C^k&0\\0&0
    \end{bmatrix}P^{-1}+\\
    &~~~~P\begin{bmatrix}C^{k-1}B_1+C^{k-2}B_1C+\cdots+B_1C^{k-1}&C^{k-1}B_2+C^{k-2}B_2N+\cdots+CB_2N^{k-2}+B_2N^{k-1}\\N^{k-1}B_3+N^{k-2}B_3C+\cdots+NB_3C^{k-2}+B_3C^{k-1}&N^{k-1}B_4+N^{k-2}B_4N+\cdots+NB_4N^{k-2}+B_4C^{k-1}
\end{bmatrix}P^{-1}\Bigg)\\&
   ~~~~\left(P\begin{bmatrix}
C^{-1}&0\\0&0
\end{bmatrix}P^{-1}+\epsilon P\begin{bmatrix}
-C^{-1}B_1C^{-1}& C^{-2}B_2+C^{-3}B_2N+\cdots+C^{-(k+1)}N^{k-1}\\B_3C^{-2}+NB_3C^{-3}+\cdots+N^{k-1}B_3C^{-(k+1)}&0
\end{bmatrix}P^{-1}\right)\\
&~~~~\left(P\begin{bmatrix}C &0\\
    0 &N\end{bmatrix}P^{-1}+\epsilon P\begin{bmatrix}B_1 &B_2\\B_3 &B_4\end{bmatrix}P^{-1}\right)\\
&=\Bigg(P\begin{bmatrix}
    C^k&0\\0&0
    \end{bmatrix}P^{-1}+\\
    &~~~~P\begin{bmatrix}C^{k-1}B_1+C^{k-2}B_1C+\cdots+B_1C^{k-1}&C^{k-1}B_2+C^{k-2}B_2N+\cdots+CB_2N^{k-2}+B_2N^{k-1}\\N^{k-1}B_3+N^{k-2}B_3C+\cdots+NB_3C^{k-2}+B_3C^{k-1}&N^{k-1}B_4+N^{k-2}B_4N+\cdots+NB_4N^{k-2}+B_4C^{k-1}
\end{bmatrix}P^{-1}\Bigg)\\
&~~\left(P\begin{bmatrix}
I&0\\0&0
\end{bmatrix}P^{-1}+\epsilon P\begin{bmatrix}
0&C^{-1}B_2+C^{-2}B_2N+\cdots+C^{-k}N^{k-1}\\B_3C^{-1}+NB_3C^{-2}+N^2B_3C^{-3}+\cdots+N^{k-1}B_3C^{-k}&0\end{bmatrix}P^{-1}\right)
\end{align*}}}
{\scriptsize{\begin{align*}
&=P\begin{bmatrix}
    C^k&0\\0&0
    \end{bmatrix}P^{-1}+\\
    &~~~~P\begin{bmatrix}C^{k-1}B_1+C^{k-2}B_1C+\cdots+B_1C^{k-1}&C^{k-1}B_2+C^{k-2}B_2N+\cdots+CB_2N^{k-2}+B_2N^{k-1}\\N^{k-1}B_3+N^{k-2}B_3C+\cdots+NB_3C^{k-2}+B_3C^{k-1}&0
    \end{bmatrix}P^{-1}\\
    &=\widehat{A}^k.
\end{align*}}}
Hence, $\widehat{A}^D=\widehat{X}$. Furthermore,
{\scriptsize{\begin{align}-A^DBA^D&=P\begin{bmatrix}-C^{-1}B_1C^{-1}&0\\0&0\end{bmatrix}P^{-1},\\
(A^D)^{k+1}D(I-AA^D)&=P\begin{bmatrix}0&C^{-2}B_2+C^{-3}B_2N+\cdots+C^{-k}B_2N^{k-2}+C^{-(k+1)}B_2N^{k-1}\\0&0
\end{bmatrix}P^{-1},\\
(I-AA^D)D(A^D)^{(k+1)}&=P\begin{bmatrix}
0& 0\\B_3C^{-2}+NB_3C^{-3}+\cdots+N^{k-1}B_3C^{-(k+1)}&0
\end{bmatrix}P^{-1}.
\end{align}}}

\noindent From (3.7), (3.8) and (3.9), we get $\widehat{X}=A^D+\epsilon[-A^DBA^D+(A^D)^{(k+1)}D(I- AA^D) + (I-AA^D)D(A^D)^{(k+1)}].$\\
\noindent (iii)$\iff$(iv):   
 By Theorem \ref{kp1}, $rank\begin{bmatrix}D &A^k\\A^k &0\end{bmatrix}=2rank(A^k)+rank(I-AA^D)D(I-AA^D)$ but $(I-AA^D)D(I-AA^D)=0$ by (iii). Conversely, if $rank\begin{bmatrix}D &A^k\\A^k &0\end{bmatrix}=2rank(A^k)$, then $rank((I-AA^D)D(I-AA^D))=0$. So, $(I-AA^D)D(I-AA^D)=0$.\\
 (iv)$\iff$(v): The equivalence of (iv) and (v) follows directly from Theorem \ref{kp2}.

\end{proof}
\begin{example}
        Let $\widehat{A}=\begin{bmatrix}
            1&1&0\\0&0&1\\0&0&0
    \end{bmatrix}+\epsilon\begin{bmatrix}
        1&2&0\\2&1&0\\0&0&1
 \end{bmatrix}$ be a  dual matrix of the form $\widehat{A}=A+\epsilon B$ with $Ind(A)=2$. Then, $A^{D}=\begin{bmatrix}
     1&1&1\\0&0&0\\0&0&0
 \end{bmatrix}$ and $R = -A^DBA^D+(A^D)^{(k+1)}D(I- AA^D) + (I-AA^D)D(A^D)^{(k+1)}=\begin{bmatrix}
     -5&-5&-7\\2&2&2\\0&0&0
 \end{bmatrix}.$ So, $\widehat{A}^D=\begin{bmatrix}
     1&1&1\\0&0&0\\0&0&0
 \end{bmatrix}+
 \epsilon\begin{bmatrix}
     -5&-5&-7\\2&2&2\\0&0&0
 \end{bmatrix}$.
\end{example}

The following corollary directly follows from the above theorem.
  \begin{corollary}
Let $\widehat{A} = A + \epsilon B$ be a dual matrix and     $D=(A^{k-1}B+A^{k-2}BA+\cdots+ABA^{k-2}+BA^{k-1})$, where $Ind(\widehat{A})=k$. If $\widehat{A}^D$ exists and $AA^DD=DAA^D=D$, then
$$\widehat{A}^D=A^D-\epsilon A^DDA^D.$$
\end{corollary}
Next example demonstrates  the above corollary.
\begin{example}
Let $\widehat{A}=A+\epsilon B=\begin{bmatrix}
-1&-1&0\\1&1&0\\0&0&0
\end{bmatrix}+\epsilon \begin{bmatrix}
0&0&0\\0&0&0\\0&0&1
\end{bmatrix}$. It is clear that $k=2$ and $D=AB+BA=0$. So, $AA^DD=DAA^D=D$. Then, $A^DDA^D=0$ and $\widehat{A}^D=A^D+\epsilon 0=A^D$.
\end{example}
% \begin{corollary}
% Let $\widehat{A} = A + \epsilon B$ be a dual matrix. If $\widehat{A}^D$ exists, then
% $$\widehat{A}^D=\begin{bmatrix}I &0\end{bmatrix}\begin{bmatrix}A &D\\0 &A\end{bmatrix}^{D}\begin{bmatrix}I\\ 0\end{bmatrix}.$$
% \end{corollary}
In 2014, Malik and Thome \cite{malika} proposed a new generalized inverse  called the DMP inverse for square matrices. Motivated by this work, we define a new generalized inverse for dual matrices that call dual  Drazin Moore-Penrose generalized inverse (DDMPGI). 
\begin{definition} If the DDGI and DMPGI of dual matrix $\widehat{A}$ exists, then
the dual  Drazin Moore-Penrose generalized inverse (DDMPGI) of dual matrices $\widehat{A}$ is
\begin{align*}
    \widehat{A}^{D,\dagger}&=\widehat{A}^D\widehat{A}\widehat{A}^{\dagger}\\
    &= [A^D+\epsilon (-A^DDA^D+(A^D)^2D(I- AA^D) + (I-AA^D)D(A^D)^2)][A+\\
    &~~~~~\epsilon B][A^{\dagger}- \epsilon (A^{\dagger}BA^{\dagger}-(A^TA)^{\dagger}B^T(I_m-AA^{\dagger})-(I_n-A^{\dagger}A)B^T(AA^T)^{\dagger})]\\
    &=[A^D+\epsilon (-A^DDA^D+(A^D)^2D(I- AA^D) + (I-AA^D)D(A^D)^2)][AA^{\dagger}+\epsilon(BA^{\dagger}\\
    &~~~~~-AA^{\dagger}BA^{\dagger}+A(A^TA)^{\dagger}B^T(I_m-AA^{\dagger}
    ))]\\
    &=A^DAA^{\dagger}+\epsilon [A^D(BA^{\dagger}-AA^{\dagger}BA^{\dagger}+A(A^TA)^{\dagger}B^T(I_m-AA^{\dagger}))+(-A^DDA^D+(A^D)^2\\&~~~~~D(I- AA^D) + (I-AA^D)D(A^D)^2)AA^{\dagger}].
\end{align*}
\end{definition}
\begin{remark}
If $\widehat{A}^{\dagger}=\widehat{A}^{D}$, then $ \widehat{A}^{D,\dagger}=\widehat{A}^D$.
\end{remark}

\section{DDGI solution of the dual linear equation $\widehat{A}\widehat{x} =\widehat{b}$}\label{rm3}
%This section presents the application of the DDGI.
In this section, we deal with finding a solution of a dual linear system  
\begin{equation}\label{eqn4.1}
    \widehat{A}\widehat{x} =\widehat{b},
\end{equation}
via the DDGI.
\begin{theorem}
If the DDGI of $\widehat{A}$ exists with $Ind(\widehat{A})=k$, then equation \eqref{eqn4.1} is consistent iff $\widehat{A}\widehat{A}^D\widehat{b}=\widehat{b}$. Furthermore, the general solution of \eqref{eqn4.1} is
\begin{equation*}
    \widehat{x}=\widehat{A}^D\widehat{b}+(\widehat{A}^{k-1}-\widehat{A}^D\widehat{A}^k)\widehat{z},
\end{equation*}
where $\widehat{z}\in \mathbb{D}^n$  and $\widehat{\mathbb{D}}$ denotes the set of dual vectors.
\end{theorem}

First, we define the range and null space of a dual matrix. We have $\widehat{A}=A+\epsilon B$, then
$$\widehat{A}^k=A^k+\epsilon(A^{k-1}B+A^{k-2}BA+\cdots+ABA^{k-2}+BA^{k-1}).$$
Now,
\begin{align*}
    R(\widehat{A}^k)&=\{\widehat{w}\in \mathbb{D}^n: ~\widehat{w}=\widehat{A}^k\widehat{z},~ \widehat{z}\in \mathbb{D}^n\}\\
    &=\{A^kx+\epsilon(A^ky+(A^{k-1}B+A^{k-2}BA+\cdots+ABA^{k-2}+BA^{k-1})x): ~ x,y\in \mathbb{R}^n\},
\end{align*}
and 
\begin{align*}
    N(\widehat{A}^k)&=\{\widehat{w}\in \mathbb{D}^n:~ 0=\widehat{A}^k\widehat{z}, ~\widehat{z}\in \mathbb{D}^n\}\\
    &=\{x+\epsilon y :~A^kx=0,A^ky+(A^{k-1}B+A^{k-2}BA+\cdots+ABA^{k-2}+BA^{k-1})x=0,  x,y\in \mathbb{R}^n\}.
\end{align*}
Further, if the DDGI exists, then $R(\widehat{A})=R(\widehat{A}^D)$ (it is easy to see by DDGI properties). The Drazin inverse $A^D$ of a matrix $A$ satisfies the following properties:
$AA^D=A^DA,~A^DAA^D=A^D,~ A^{k+1}A^D=A^k$. So, we have $A(A^D)^2=A^D$. Then, 
\begin{align*}
    AA^D&=A(A(A^D)^2)
    =A^2(A^D)^2\\
    &=A^2(A(A^D)^2)A^D
    =A^3(A^D)^3\\
    &=A^k(A^D)^k
    =(A^D)^kA^k.
\end{align*}
%i.e., $\widehat{A}^{k+1}(\widehat{z}-\widehat{A}^Db)=0$ implies $(\widehat{A}^D)\widehat{A}^{k+1}(\widehat{z}-\widehat{A}^Db)=0$, i.e.,
Using the discussion made before, we now present the primary result of this section.
%$\in \mathbb{D}^{n\times n}$
\begin{theorem}
If the DDGI of a dual matrix $\widehat{A}=A+\epsilon B$ exists with $Ind(\widehat{A})=k$,
then $\widehat{A}^Db$ is the unique solution of
\begin{equation}\label{eq4.1}
    \widehat{A}\widehat{x}=\widehat{b},~ \widehat{x}\in R(\widehat{A}^k).
\end{equation}
\end{theorem}
\begin{proof}
Firstly, if $\widehat{A}^D$ exists, then \eqref{eq4.1} is consistent and $\widehat{A}^Db$ is a solution.\\
We know that $R(\widehat{A}^D)=R(\widehat{A}^k)$. And $\widehat{A}^Db\in R(\widehat{A}^D)=R(\widehat{A}^k)$. Let $\widehat{z}$ be the another solution of this system. So, $\widehat{z}\in R(\widehat{A}^D)=R(\widehat{A}^k)$. Hence, $\widehat{z}-\widehat{A}^Db\in R(\widehat{A}^D)$. And $\widehat{A}^{}(\widehat{z}-\widehat{A}^Db)=0$, i.e., $\widehat{A}^k(\widehat{z}-\widehat{A}^Db)=0$ implies $(\widehat{z}-\widehat{A}^Db)\in N(\widehat{A}^k).$ So, $(\widehat{z}-\widehat{A}^Db)\in N(\widehat{A}^k)\cap R(\widehat{A}^k)$. \\
We know that if $A^D$ exists, then $R(A^k)\cap N(A^k)=\{0\}$.  We have to prove $N(\widehat{A}^k)\cap R(\widehat{A}^k)=\{0\}$. Hence, $\widehat{z}-\widehat{A}^Db=0$, i.e., $\widehat{z}=\widehat{A}^Db.$ So, $\widehat{A}^Db$ is unique.\\
%For any $\widehat{y}\in N(\widehat{A}^k)\cap R(\widehat{A}^k)$, 
Let $\widehat{l}\in R(\widehat{A}^k)\cap N(\widehat{A}^k)$, there exist two vectors $m$ and $n$ such that $\widehat{l}=A^km+\epsilon(A^kn+(A^{k-1}B+A^{k-2}BA+\cdots+ABA^{k-2}+BA^{k-1})m)$. We know that $\widehat{A}^k\widehat{l}=0$, i.e.,
\begin{align*}
 [A^k+\epsilon(A^{k-1}B+A^{k-2}BA+\cdots+ABA^{k-2}+BA^{k-1})] [A^km+\epsilon(A^kn+(A^{k-1}B+A^{k-2}BA\\
   ~~+\cdots+ABA^{k-2}+BA^{k-1})m)]=0,
   \end{align*}
\text {i.e.,}
\begin{align*} A^{2k}m+\epsilon(A^{2k}n+A^k(A^{k-1}B+A^{k-2}BA
   ~~+\cdots+ABA^{k-2}+BA^{k-1})m+(A^{k-1}B+A^{k-2}BA\\
   +\cdots+ABA^{k-2}+BA^{k-1})A^km)=0.
\end{align*}
Hence, $A^{2k}m=0$ and $A^{2k}n+A^k(A^{k-1}B+A^{k-2}BA
   ~~+\cdots+ABA^{k-2}+BA^{k-1})m+(A^{k-1}B+A^{k-2}BA
   +\cdots+ABA^{k-2}+BA^{k-1})A^km=0$.\\
So, $A^k(A^km)=0$. Thus, $A^km\in R(A^k)\cap N(A^k)=\{0\}$. Further, we obtain  $A^{2k}n+A^k(A^{k-1}B+A^{k-2}BA+\cdots+ABA^{k-2}+BA^{k-1})m=0$, i.e., $A^k(A^{k}n+(A^{k-1}B+A^{k-2}BA+\cdots+ABA^{k-2}+BA^{k-1})m)=0$, which implies $A^{k}n+(A^{k-1}B+A^{k-2}BA+\cdots+ABA^{k-2}+BA^{k-1})m\in N(A^k)$. We have $\widehat{A}^D$ is exists, then by \eqref{eqn3.1},
\begin{align}\label{eq4.4}
    A^{k-1}B+A^{k-2}BA+&\cdots+ABA^{k-2}+BA^{k-1}\nonumber\\
    &=A^kA^DB+A^kRA+(A^{k-1}B+A^{k-2}BA+\cdots+ABA^{k-2}+BA^{k-1})A^DA.
\end{align}
Using equation \eqref{eq4.4}, we get 
\begin{align*}
    (A^{k}n+&(A^{k-1}B+A^{k-2}BA+\cdots+ABA^{k-2}+BA^{k-1})m)\\
    &= A^kn+ [A^kA^DB+A^kRA+(A^{k-1}B+A^{k-2}BA+\cdots+ABA^{k-2}+BA^{k-1})(A^DA)]m\\
    &=A^kn+ [A^kA^DB+A^kRA+(A^{k-1}B+A^{k-2}BA+\cdots+ABA^{k-2}+BA^{k-1})(A^D)^k(A^k)]m\\
     &=A^kn+ A^kA^DBm+A^kRAm\\
     &=A^k(In+A^DBm+RAm),
\end{align*}
which implies that     $A^{k}n+(A^{k-1}B+A^{k-2}BA+\cdots+ABA^{k-2}+BA^{k-1})m\in R(A^k)$. Hence, $    A^{k}n+(A^{k-1}B+A^{k-2}BA+\cdots+ABA^{k-2}+BA^{k-1})m\in R(A^k)\cap N(A^k)=\{0\}$. Therefore, $l=0$, i.e.,   $R(\widehat{A}^k)\cap N(\widehat{A}^k)=0 $, i.e., $\widehat{z}-\widehat{A}^Db=0$, i.e., $\widehat{z}=\widehat{A}^Db.$ So, $\widehat{A}^Db$ is unique.\\
% $$(A^k+\epsilon B^k)(A^km+\epsilon (A^kn+B^km))=A^{2k}m+\epsilon [A^{2k}n+(A^kB^k+B^kA^k)m]=0.$$
% $A^{2k}x=A^k(A^km)=0$, that $A^km\in R(A^k)\cap N(A^k)=\{0\}. $ So, $A^km=0$. And $[A^{2k}n+(A^kb^k+B^kA^k)m]=0$, i.e., $0=A^{2k}n+A^kB^km=A^k(A^kn+B^km)$ implies $(A^kn+B^km)\in  R(A^k)\cap N(A^k) $. We know that $\widehat{A}^D$ exists, then by Theorem \ref{}.  
% %$$(I-AA^D)(A^{k-1}B+A^{k-2}BA+\cdot%s+ABA^{k-2}+BA^{k-1})(A^DA-I)=0.$$
% We have $R(B)\subseteq R(A^k)$. Thus for every $x\in R$ there exists $y\in R^n$ such that $Bx=A^ky$. So,
% $$A^kn+Bm=A^kn+A^ks=A^k(n+s)\in R(A^k).$$
% Hence, 
% $(A^kn+B^km)\in  R(A^k)\cap N(A^k) $, i.e., $(A^kn+B^km)=0$. 
\end{proof}
\section{Reverse and forward order laws for particular form}\label{rm4}
In this section, we present the reverse and the forward-order laws for  DMPGI, DGGI, DDGI and DGCI for particular form $\widehat{A}^c=A^c-\epsilon A^cAA^c$, where $A^c$ is the notation used for MP inverse, group inverse, Drazin inverse, and core inverse. We start this section with an example which show that reverse and forward-order laws do not always hold for the DGGI.
\begin{example}

Let $\widehat{A}=\begin{bmatrix}
2&1&3\\0&0&0\\1&1&2
\end{bmatrix}+\epsilon \begin{bmatrix}
2&2&4\\3&-1&2\\-4&-2&-6
\end{bmatrix}$ and $\widehat{B}=\begin{bmatrix}
1&-1&0\\0&0&0\\-1&3&2
\end{bmatrix}+\epsilon\begin{bmatrix}
2&-4&3\\0&0&0\\1&-5&6
\end{bmatrix}$. 
Then, we get  $\widehat{A}^{\#}=\begin{bmatrix}2&-5&-3\\0&0&0\\-1&3&2
\end{bmatrix}+\epsilon \begin{bmatrix}
27&-78&-51\\13&-35&-22\\-21&60&39
\end{bmatrix}$, $\widehat{B}^{\#}=\begin{bmatrix}
1&-1&0\\0&0&0\\-1/3&1/9&1/3
\end{bmatrix}+\epsilon \begin{bmatrix}
1&-1.6667&1\\0&0&0\\-.6667&.4444&.3333
\end{bmatrix}$ and 
$\widehat{A}\widehat{B}=\begin{bmatrix}
5&-11&9\\0&0&0\\3&-7&6
\end{bmatrix}+\epsilon \begin{bmatrix}
13& -37&36\\5&-9&6\\-6&8&-3
\end{bmatrix}$. Therefore, 
$(\widehat{A}\widehat{B})^{\#}=\begin{bmatrix}
2&0&-3\\0&0&0\\-1&-1/9&5/3
\end{bmatrix}+\epsilon \begin{bmatrix}
6&86/9&-70/3\\13&5/9&-61/3\\
127/9&-115/27&-133/9
\end{bmatrix}$,\\
$\widehat{A}^{\#}\widehat{B}^{\#}=\begin{bmatrix}
2.9999&-2.3333&-0.9999\\
0&0&0\\
.0001&.6666&.9996
\end{bmatrix}+\epsilon \begin{bmatrix}
47.9984&-37.3327&-15.9982\\20.3326&-15.4442&-7.3326\\-34.9988&24.9994&14.9986
\end{bmatrix}$ and
$\widehat{B}^{\#}\widehat{A}^{\#}=\begin{bmatrix}
2&-5&-3\\0&0&0\\-.3333&.6666&1.9998
\end{bmatrix}+\epsilon \begin{bmatrix}
17&-51&-29\\0&0&0\\-15.5542&44.4485&30.5528
\end{bmatrix}$. It is clear that $(\widehat{A}\widehat{B})^{\#}\neq \widehat{A}^{\#}\widehat{B}^{\#}\neq\widehat{B}^{\#}\widehat{A}^{\#}.$
\end{example}
Similarly, the reverse-order and forward-order laws do not hold for DMPGI, DCGI, or DDGI. We will start with the result that provides sufficient conditions for reverse and forward-order laws.
\begin{theorem}
Let $\widehat{A}=A+\epsilon B$ and $\widehat{C}=C+\epsilon D$ be such that the DGGI  of $\widehat{A}$, $\widehat{C}$, $\widehat{A}\widehat{C}$ exist. If $AC=CA$, $C^{\#}B=BC^{\#}$ and  $A^{\#}D=DA^{\#}$, then $(\widehat{A}\widehat{C})^{\#}=\widehat{A}^{\#}\widehat{C}^{\#}=\widehat{C}^{\#}\widehat{A}^{\#}$.
\end{theorem}
\begin{proof} The DGGI of $\widehat{A}$ and $\widehat{C}$ are 
$\widehat{A}^{\#}=A^{\#}-\epsilon A^{\#}BA^{\#}$ and $\widehat{C}^{\#}=C^{\#}-\epsilon C^{\#}DC^{\#}$, respectively. So,  $\widehat{A}^{\#}\widehat{C}^{\#}=(A^{\#}-\epsilon A^{\#}BA^{\#})(C^{\#}-\epsilon C^{\#}DC^{\#})=A^{\#}C^{\#}-\epsilon (A^{\#}C^{\#}DC^{\#}+A^{\#}BA^{\#}C^{\#}).$ And $\widehat{A}\widehat{C}=(A+\epsilon B)(C+\epsilon D)=AB+\epsilon (AD+BC).$ By Corollary, $(\widehat{A}\widehat{C})^{\#}=(AC)^{\#}-\epsilon ((AC)^{\#}(AD+BC)(AC)^{\#})$. The equality $AC=CA$ implies that  $(AC)^{\#}=A^{\#}C^{\#}=C^{\#}A^{\#}$. Further, $(\widehat{A}\widehat{C})^{\#}=(AC)^{\#}-\epsilon ((AC)^{\#}(AD+BC)(AC)^{\#})=A^{\#}C^{\#}-\epsilon (A^{\#}C^{\#}(AD+BC)A^{\#}C^{\#})=A^{\#}C^{\#}-\epsilon (A^{\#}C^{\#}ADA^{\#}C^{\#}+A^{\#}C^{\#}BCA^{\#}C^{\#})=A^{\#}C^{\#}-\epsilon (C^{\#}A^{\#}AA^{\#}DC^{\#}+A^{\#}BC^{\#}CC^{\#}A^{\#})=A^{\#}C^{\#}-\epsilon (C^{\#}A^{\#}DC^{\#}+A^{\#}BC^{\#}A^{\#})=A^{\#}C^{\#}-\epsilon (A^{\#}C^{\#}DC^{\#}+A^{\#}BC^{\#}A^{\#})=\widehat{A}^{\#}\widehat{C}^{\#}$. 
 Similarly, $(\widehat{A}\widehat{C})^{\#}=\widehat{C}^{\#}\widehat{A}^{\#}$. Hence, $(\widehat{A}\widehat{C})^{\#}=\widehat{A}^{\#}\widehat{C}^{\#}=\widehat{C}^{\#}\widehat{A}^{\#}$.
 \end{proof}
The existence of  reverse and forward-order laws for the DDGI  can be proved similar to the above theorem and is stated below.
\begin{theorem}
Let $\widehat{A}=A+\epsilon B$ and $\widehat{C}=C+\epsilon D\in\mathbb{D}^{n\times n}$ be such that  the DDGI  of $\widehat{A}$, $\widehat{C}$, $\widehat{A}\widehat{C}$ exist. If $AC=CA$, $C^{D}B=BC^{D}$ and  $A^{D}D=DA^{D}$, then $(\widehat{A}\widehat{C})^{D}=\widehat{A}^{D}\widehat{C}^{D}=\widehat{C}^{D}\widehat{A}^{D}$.
\end{theorem}
Next result establishes the reverse and forward-order laws for the DMPGI.
\begin{theorem}
Let $\widehat{A}=A+\epsilon B$ and $\widehat{C}=C+\epsilon D$ be such that the DMPGI of $\widehat{A}$, $\widehat{C}$, $\widehat{A}\widehat{C}$ exist. If $AC=CA$, $A^*C=CA^*$, $C^{\dagger}B=BC^{\dagger}$ and  $A^{\dagger}D=DA^{\dagger}$, then $(\widehat{A}\widehat{C})^{\dagger}=\widehat{A}^{\dagger}\widehat{C}^{\dagger}=\widehat{C}^{\dagger}\widehat{A}^{\dagger}$.
\end{theorem}
\begin{proof} The DMPGI of $\widehat{A}$ and $\widehat{C}$ are 
$\widehat{A}^{\dagger}=A^{\dagger}-\epsilon A^{\dagger}BA^{\dagger}$ and $\widehat{C}^{\dagger}=E^{\dagger}-\epsilon C^{\dagger}DC^{\dagger}$, respectively. So,  $\widehat{A}^{\dagger}\widehat{C}^{\dagger}=(A^{\dagger}-\epsilon A^{\dagger}BA^{\dagger})(C^{\dagger}-\epsilon C^{\dagger}DC^{\dagger})=A^{\dagger}C^{\dagger}-\epsilon (A^{\dagger}C^{\dagger}DC^{\dagger}+A^{\dagger}BA^{\dagger}C^{\dagger}).$ And $\widehat{G}\widehat{Y}=(A+\epsilon B)(C+\epsilon D)=AB+\epsilon (AD+BC).$ By Corollary, $(\widehat{A}\widehat{C})^{\dagger}=(AC)^{\dagger}-\epsilon ((AC)^{\dagger}(AD+BC)(AC)^{\dagger})$. The hypothesis $AC=CA$ and $A^*C=CA^*$ imply that  $(AC)^{\dagger}=A^{\dagger}C^{\dagger}=C^{\dagger}A^{\dagger}$. Further, $(\widehat{A}\widehat{C})^{\dagger}=(AC)^{\dagger}-\epsilon ((AC)^{\dagger}(AD+BC)(AC)^{\dagger})=A^{\dagger}C^{\dagger}-\epsilon (A^{\dagger}C^{\dagger}(AD+BC)A^{\dagger}C^{\dagger})=A^{\dagger}C^{\dagger}-\epsilon (A^{\dagger}C^{\dagger}ADA^{\dagger}C^{\dagger}+A^{\dagger}C^{\dagger}BCA^{\dagger}C^{\dagger})=A^{\dagger}C^{\dagger}-\epsilon (C^{\dagger}A^{\dagger}AA^{\dagger}DC^{\dagger}+A^{\dagger}BC^{\dagger}CC^{\dagger}A^{\dagger})=A^{\dagger}C^{\dagger}-\epsilon (C^{\dagger}A^{\dagger}DC^{\dagger}+A^{\dagger}BC^{\dagger}A^{\dagger})=A^{\dagger}C^{\dagger}-\epsilon (A^{\dagger}C^{\dagger}DC^{\dagger}+A^{\dagger}BC^{\dagger}A^{\dagger})=\widehat{A}^{\dagger}\widehat{C}^{\dagger}$. 
 Similarly, $(\widehat{A}\widehat{C})^{\dagger}=\widehat{C}^{\dagger}\widehat{A}^{\dagger}$. Hence, $(\widehat{A}\widehat{C})^{\dagger}=\widehat{A}^{\dagger}\widehat{C}^{\dagger}=\widehat{C}^{\dagger}\widehat{A}^{\dagger}$.
 \end{proof}
The proof of the reverse and forward-order laws for the DCGI is similar to the above theorem. So, we state it without its proofs.
\begin{theorem}
Let $\widehat{A}=A+\epsilon B$ and $\widehat{C}=C+\epsilon D$ be such that  the DCGI of $\widehat{A}$, $\widehat{C}$, $\widehat{A}\widehat{C}$ exist. If $AC=CA$, $A^*C=CA^*$, $\core{C}B=B\core{C}$ and  $\core{A}D=D\core{A}$, then $\core{(\widehat{A}\widehat{C})}=\core{\widehat{A}}\core{\widehat{C}}=\core{\widehat{C}}\core{\widehat{A}}$.
\end{theorem}
At the of this section, we give a result of the absorption law for the DDGI,
\begin{theorem}
Let $\widehat{A}=A+\epsilon B$ and
$\widehat{C}=C+\epsilon D$ be such that the DDGI  of $\widehat{A}$, $\widehat{C}$ exist. If $A=D,$ $R(A)=R(C)$ and $N(A)=N(C)$, then $\widehat{A}^{D}(\widehat{A}+\widehat{C})\widehat{C}^{D}=\widehat{A}^{D}+\widehat{C}^D$.
\end{theorem}
\begin{proof} The hypothesis $R(A)=R(C)$ and $N(A)=N(C)$ imply that $A^DAE^D=A^D=E^DAA^D$ and $EE^DA^D=A^DEE^D$. Then,
    \begin{align*}
        \widehat{A}^{D}(\widehat{A}+\widehat{C})\widehat{C}^{D}&=
    (A^D-\epsilon A^DBA^D)(A+\epsilon B+C+\epsilon D)(C^D-\epsilon C^DDC^D)\\
    &=(A^DA+A^DC+\epsilon(A^DB+A^DD-A^DBA^DA-A^DBA^DD)(C^D-C^DDC^D)\\
    &=A^DAC^D+A^DCC^D+\epsilon(A^DBC^D+A^DDC^D-A^DBA^DAC^D\\&-A^DBA^DDC^D-A^DAC^DDC^D-A^DCC^D DC^D)\\
    &=A^D+C^D-\epsilon(A^DBA^DDC^D-A^DAC^DDC^D)\\
    &=A^D-\epsilon A^DBA^DDC+C^D-\epsilon A^DAC^DC^D\\
    &=A^D-\epsilon A^DBA^DDC^D+C^D-\epsilon C^DDC^D\\
     &=A^D-\epsilon A^DBA^D+C^D-\epsilon C^DDC^D\\
     &=\widehat{A}+\widehat{C}.
        \end{align*}
\end{proof}
\section{Partial order of DGGI, DDGI and DCGI}
This section is devoted for partial ordering based on DGGI, DDGI and DCGI.

\begin{definition}
 Let $\widehat{X}=X+\epsilon X_0$ and $ \widehat{Y}=Y+\epsilon{Y_0}$ be such that  the DGGI of $\widehat{X}$ exists. Then, $\widehat{X}$ is said to be below $\widehat{Y}$ under a D-group  order if 
 $$\widehat{X}^{\#}\widehat{X}=\widehat{X}^{\#}\widehat{Y}, \text{ and } X\widehat{X}^{\#}=\widehat{Y}\widehat{X}^{\#}.$$
 It is denoted by $\widehat{X}\leq_{D}^{\#}\widehat{Y}$.
\end{definition}
Next theorem relates the D-group order and group partial order.
\begin{theorem}\label{thma21}
    Let $\widehat{X}=X+\epsilon X_0$ and $ \widehat{Y}=Y+\epsilon{Y_0}$ be such that  the DGGI of $\widehat{X}$ exists. Then, $X\leq_{D}^{\#}Y$ if and only if 
    \begin{align*}
        X&\leq^{\#}Y,\\
          X^{\#}X_0+RX&= X^{\#}Y_0+RY,\\%&=YR+Y_0X^{\#},\\
          XR+X_0X^{\#}&= YR+Y_0X^{\#},
    \end{align*}
    where $R=-X^{\#}X_0X^{\#}+(X^{\#})^2X_0(I-XX^{\#})+(I-XX^{\#})X_0(X^{\#})^2.$
% $$\widehat{X}^{\#}\widehat{Y}=\widehat{X}^{\#}\widehat{Y}, \text{ and } X\widehat{X}^{\#}=\widehat{Y}\widehat{X}^{\#}.$$  It is denoted by 
\end{theorem}
\begin{proof}Suppose that $X\leq_{D}^{\#}Y$, i.e., $\widehat{X}^{\#}\widehat{X}=\widehat{X}^{\#}\widehat{Y}, \text{ and } X\widehat{X}^{\#}=\widehat{Y}\widehat{X}^{\#}$. Now,
\begin{align}\label{par1}
    \widehat{X}^{\#}\widehat{X}&=(X^{\#}+\epsilon R)(X+\epsilon X_0)\nonumber\\
    &=X^{\#}X+\epsilon(X^{\#}X_0+RX)
\end{align}
and 
\begin{align}\label{par2}
      \widehat{X}^{\#}\widehat{Y}&=(X^{\#}+\epsilon R)(Y+\epsilon Y_0)\nonumber\\
    &=X^{\#}Y+\epsilon(X^{\#}Y_0+RY).
\end{align}
% and
% \begin{align}\label{par2}
%     \widehat{Y}\widehat{X}^{\#}&=(Y+\epsilon Y_0)(X^{\#}+\epsilon R)\nonumber\\
%     &=YX^{\#}+\epsilon (YR+Y_0X^{\#}).
% \end{align}
From \eqref{par1} and \eqref{par2}, we get 
\begin{align*}
    X^{\#}X&=X^{\#}Y\\
   X^{\#}X_0+RX&= X^{\#}Y_0+RY.
\end{align*}
Similarly, the dual matrix equality $\widehat{Y}\widehat{X}^{\#}=\widehat{X}\widehat{X}^{\#}$ implies that
\begin{align*}
    XX^{\#}&=YX^{\#}\\
    XR+X_0X^{\#}&=YR+Y_0X^{\#}.
\end{align*}
So, we get  
  \begin{align*}
        X^{\#}X=X^{\#}Y~ &~\&XX^{\#}=YX^{\#}\\
          X^{\#}X_0+RX&= X^{\#}Y_0+RY,\\%&=YR+Y_0X^{\#},\\
          XR+X_0X^{\#}&= YR+Y_0X^{\#},
    \end{align*}
    i.e., 
      \begin{align*}
        X&\leq^{\#}Y,\\
          X^{\#}X_0+RX&= X^{\#}Y_0+RY,\\%&=YR+Y_0X^{\#},\\
          XR+X_0X^{\#}&= YR+Y_0X^{\#}.
    \end{align*}
\end{proof}
Groß \cite{gro} established a representation of matrices with the group partial order.
\begin{theorem}(\cite{gro})\label{thm221}
    Let $A,B\in\mathbb{R}^{n\times n}$ with rank$A=r$ and $Ind(A)=1$. Then, the following statements are equivalents
    \begin{enumerate}
        \item[(i)] $A\leq^{\#}B$;
        \item[(ii)] $A$ and $B$ can be written as 
        $$A=P\begin{bmatrix}
            C&0\\0&0
        \end{bmatrix}P^{-1} \text{ and }B=P\begin{bmatrix}
            C&0\\0&B_1
        \end{bmatrix}P^{-1},$$
    \end{enumerate}
    where $C\in\mathbb{R}^{r\times r}$ and $P\in\mathbb{R}^{n\times n}$ are nonsingular matrix, and $B_1\in\mathbb{R}^{n-r\times n-r}$ is arbitrary.
\end{theorem}
We now provide representations the dual matrices with the help of D-group order.
\begin{theorem}\label{th22}
    Let $\widehat{X}=X+\epsilon X_0$ and $ \widehat{Y}=Y+\epsilon{Y_0}$ be such that  the DGGI of $\widehat{X}$ exists. If $X\leq_{D}^{\#}Y$, then 
   \begin{align*}
       \widehat{X}&=P\begin{bmatrix}
           C&0\\0&0
       \end{bmatrix}P^{-1}+\epsilon P\begin{bmatrix}
           X_1&X_2\\X_3&0
       \end{bmatrix}P^{-1},\\
       \widehat{Y}&=P\begin{bmatrix}
            C&0\\0&B
        \end{bmatrix}P^{-1}+\epsilon P\begin{bmatrix}
            X_1&X_2-CR_2B\\X_3-BR_3C&Y_4
        \end{bmatrix}P^{-1},
   \end{align*}
   where $Y_4$ is arbitrary matrix, and $\widehat{X}^{\#}=X^{\#}+\epsilon R$, $X^{\#}=P\begin{bmatrix}
           C^{-1}&0\\0&0
       \end{bmatrix}P^{-1},$ $R=P\begin{bmatrix}
       R_1&R_2\\R_3&R_4
   \end{bmatrix}P^{-1}$. 
\end{theorem}
\begin{proof}
   Given $\widehat{X}^{\#}$ exists. So, by Theorem \ref{thm2}, we have
    \begin{align}\label{par3}
        \widehat{X}=P\begin{bmatrix}
           C&0\\0&0
       \end{bmatrix}P^{-1}+\epsilon P\begin{bmatrix}
           X_1&X_2\\X_3&0
       \end{bmatrix}P^{-1}.
    \end{align}
    From Theorem \ref{thma21} and Theorem \ref{thm221}, we get 
    \begin{align}\label{par4}
        \widehat{Y}=P\begin{bmatrix}
            C&0\\0&B
        \end{bmatrix}P^{-1}+\epsilon P\begin{bmatrix}
            Y_1&Y_2\\Y_3&Y_4
        \end{bmatrix}P^{-1}. 
    \end{align}
    Again, from Theorem \ref{thma21} and \eqref{par3}, \eqref{par4} we get 
    \begin{align*}
        \widehat{X}^{\#}\widehat{X}&=P^{-1}\begin{bmatrix}
            I_r&0\\0&0
        \end{bmatrix}P^{-1}+\epsilon P\begin{bmatrix}
            R_1C^{-1}+C^{-1}X_1&C^{-1}X_2\\R_3C^{-1}&0
        \end{bmatrix}P^{-1},\\
        \widehat{X}^{\#}\widehat{Y}&=P\begin{bmatrix}
            I_r&0\\0&0
        \end{bmatrix}P^{-1}+\epsilon P\begin{bmatrix}
            C^{-1}Y_1+R_1C&C^{-1}Y_2+R_2B\\R_3C&R_4B
        \end{bmatrix}P^{-1}\\
                \widehat{X}\widehat{X}^{\#}&=P^{-1}\begin{bmatrix}
            I_r&0\\0&0
        \end{bmatrix}P^{-1}+\epsilon P\begin{bmatrix}
            CR_1+X_1C^{-1}&CR_2\\X_3C^{-1}&0
        \end{bmatrix}P^{-1},\\
        \widehat{Y}\widehat{X}^{\#}&=P\begin{bmatrix}
            I_r&0\\0&0
        \end{bmatrix}P^{-1}+\epsilon P\begin{bmatrix}
           Y_1 C^{-1}+CR_1&CR_2\\Y_3C^{-1}+BR_3&BR_4
        \end{bmatrix}P^{-1}.      
    \end{align*}
    From the above four equations, we obtain $Y_1=X_1$, $Y_2=X_2-CR_2B$, and $Y_3=X_3-BR_3C$.
\end{proof}
\begin{theorem}
    The D-group order is a partial order.
\end{theorem}
% \begin{proof}
%      The reflexivity is obvious. Suppose that $\widehat{X}\leq^{\#}_{D}\widehat{Y}$ and $\widehat{Y}\leq^{\#}_D\widehat{X}$. Then, the Theorem \ref{th22} implies that
%   \begin{align*}
%        \widehat{X}&=P\begin{bmatrix}
%            C&0\\0&0
%        \end{bmatrix}P^{-1}+\epsilon P\begin{bmatrix}
%            C&X_2\\X_3&0
%        \end{bmatrix}P^{-1},\\
%        \widehat{Y}&=P\begin{bmatrix}
%             C&0\\0&0
%         \end{bmatrix}P^{-1}+\epsilon P\begin{bmatrix}
%             X_1&X_2\\X_3&0
%         \end{bmatrix}P^{-1}.
%    \end{align*}
%    So, $\widehat{X}=\widehat{Y}$, i.e., the D-group order is anti-symmetric. Now, we check the transitivity. Let $\widehat{X}\leq^{\#}_{D}\widehat{Y}$ and $\widehat{Y}\leq^{\#}_D\widehat{Z}$. 
%      \begin{align*}
%        \widehat{X}&=P\begin{bmatrix}
%            C&0&0\\0&0&0\\0&0&0
%        \end{bmatrix}P^{-1}+\epsilon P\begin{bmatrix}
%            X_1&X_2&X_4\\X_3&0&0\\X_5&0&0
%        \end{bmatrix}P^{-1},\\
%        \widehat{Y}&=P\begin{bmatrix}
%             C&0&0\\0&B&0\\0&0&0
%         \end{bmatrix}P^{-1}+\epsilon P\begin{bmatrix}
%             X_1&X_2-CR_2B&Y_5\\X_3-BR_3C&Y_4&Y_6\\
%             Y_7&Y_8&Y_9
%         \end{bmatrix}P^{-1}.
%    \end{align*}
   
% \end{proof}
Here, we introduce another ordering in dual matrices, it called D-core order.
\begin{definition}
 Let $\widehat{X}=X+\epsilon X_0$ and $ \widehat{Y}=Y+\epsilon{Y_0}$ be such that  the DCGI of $\widehat{X}$ exists. Then, $\widehat{X}$ is said to be below $\widehat{Y}$ under a D-core  order if 
 $$\core{\widehat{X}}\widehat{X}=\core{\widehat{X}}\widehat{Y}, \text{ and } X\core{\widehat{X}}=\widehat{Y}\core{\widehat{X}}.$$
 It is denoted by $\widehat{X}\core{\leq_{D}}\widehat{Y}$.
\end{definition}
This can be prove analogous of  Theorem \ref{thma21}.
\begin{theorem}\label{thm21}
    Let $\widehat{X}=X+\epsilon X_0$ and $ \widehat{Y}=Y+\epsilon{Y_0}$ be such that  the DGGI of $\widehat{X}$ exists. Then, $X\core{\leq_{D}}Y$ if and only if 
    \begin{align*}
        X&\core{\leq} Y,\\
          \core{X}X_0+RX&= \core{X}Y_0+RY,\\%&=YR+Y_0X^{\#},\\
          XR+X_0\core{X}&= YR+Y_0\core{X},
    \end{align*}
    where $R=-\core{X}X_0X^{\dagger}+X^{\#}X_0X^{\dagger}-X^{\#}X_0\core{X}+\core{X}(X_0X^{\dagger})^T(I-XX^{\#})+(I-XX^{\#})X_0X^{\#}\core{X}.$
% $$\widehat{X}^{\#}\widehat{Y}=\widehat{X}^{\#}\widehat{Y}, \text{ and } X\widehat{X}^{\#}=\widehat{Y}\widehat{X}^{\#}.$$  It is denoted by 
\end{theorem}
\begin{theorem}
    The D-core order is a partial order.
\end{theorem}
\section{Conclusion}
The essential determinations are summarized as follows:
\begin{itemize}
\item The necessary and sufficient conditions for the existence of the DDGI are established.
\item Application of the DDGI for solving a dual linear system is presented.
% \item The forward-order laws for the core inverse and the weighted core inverse have been introduced in rings.
 \item The reverse and forward-order laws for a particular form of  the MPDGI, DGGI, DCGI, and DDGI  have been established.
 \item Finally, a few necessary and sufficient conditions for the existence of partial-order of DCGI and DGGI have been discussed.
\end{itemize}

\section*{Data Availability Statements}
Data sharing not applicable to this article as no data sets were generated or analysed during the current study.
%Data sharing not applicable to this article as no datasets were generated or analysed during the current study.

\section*{Conflicts of interest} 

The authors declare that they have no conflict of interest.
% \section*{Funding} 
% This declaration is “not applicable”.
\section*{Acknowledgements}

The first author acknowledges the support of the CSIR-UGC, India.

% % {\bf Problem Statement:} Numerical range of Moore-Penrose inverse of dual matrices.
 
%  Let $\widehat{A}=A+\epsilon B$ and $\widehat{\lambda}=\alpha+\epsilon \beta$ and $\widehat{z}=x+\epsilon y$
%  \begin{align*}
%      (\widehat{A}-\widehat{\lambda}I)\widehat{z}&=0\\
%      [(A-\alpha I)+\epsilon (B-\beta I)](x+\epsilon y)&=0\\
%      (A-\alpha I)x+\epsilon[(A-\alpha I)y +(B-\beta I)x]&=0.
%  \end{align*}
% So, we get $ (A-\alpha I)x=0$ and $(A-\alpha I)y +(B-\beta I)x=0$.\\
% \textcolor{red}{Example of eigenvalue  and Perturbation in eigenvalue}
\end{document}